\documentclass[runningheads,a4paper]{article}

\usepackage{
amsmath,
amsthm,
amssymb,
}
\usepackage{comment}
\usepackage{tikz}
\usetikzlibrary{positioning,arrows,calc}
\usepackage{xspace}

\usepackage{stmaryrd}

\usepackage{graphicx}

\usepackage{url}

\newtheorem{lemma}{Lemma}
\newtheorem{definition}{Definition}
\newtheorem{corollary}{Corollary}
\newtheorem{proposition}{Proposition}
\newtheorem{example}{Example}
\newtheorem{theorem}{Theorem}

\newtheorem{question}{Question}

\newcommand{\Z}{\mathbb{Z}}

\newcommand{\N}{\mathbb{N}}

\newcommand{\B}{\mathcal{L}}

\newcommand{\ID}{\mathrm{id}}

\newcommand{\Aut}{\mathrm{Aut}}
\newcommand{\End}{\mathrm{End}}

\newcommand\xqed[1]{%
  \leavevmode\unskip\penalty9999 \hbox{}\nobreak\hfill
  \quad\hbox{#1}}
\newcommand\qee{\xqed{$\triangle$}}

\newcommand{\subg}{\mathrm{sub}_g}
\newcommand{\subm}{\mathrm{sub}_m}

\title{A note on subgroups of automorphism groups of full shifts}

\author{
Ville Salo \\
University of Chile
}

\begin{document}
\maketitle

\begin{abstract}
We discuss the set of subgroups of the automorphism group of a full shift, and submonoids of its endomorphism monoid. We prove closure under direct products in the monoid case, and free products in the group case. We also show that the automorphism group of a full shift embeds in that of an uncountable sofic shift. Some undecidability results are obtained as corollaries.
\end{abstract}


\section{Outline}

Automorphism groups of full shifts were first studied in \cite{He69}, and the transitive SFT case was later studied in more detail in \cite{BoLiRu88}. More results on them, and in particular on the set of subgroups of these groups, were shown in \cite{KiRo90}. In particular, interesting results were shown about the set of subgroups of $\Aut(S^\Z)$. It is known for example that free groups, finite groups and finitely generated abelian groups can be embedded in this group \cite{He69,BoLiRu88}, and so can all locally finite residually finite countable groups \cite{KiRo90}. Other examples are `graph groups'
\[ \langle (g_i)_{i \in \N} \;|\; (g_i g_j = g_j g_i)_{(i, j) \in R} \rangle \]
where $R \subset \N^2$ is arbitrary, fundamental groups of $2$-manifolds, and the free product of all finite groups \cite{KiRo90}. Despite the long list of examples, to our knowledge there is no known characterization of the set of subgroups of $\Aut(S^\Z)$.

In the absense of a characterization, we turn to closure properties. One closure property for this set of groups is proved explicitly in \cite{KiRo90}, namely closure under extensions by finite groups. Another one can be found by a direct application of the ideas of \cite{BoLiRu88,KiRo90}, namely closure under direct sums, and we show this in Theorem~\ref{thm:DirectProducts}. Our main result is that this set of groups is also closed under free products, Theorem~\ref{thm:FreeProducts}, which answers a question I myself asked in \cite{Sa15}.

We also show $\Aut(S^\Z)$ is contained in the automorphism group of every uncountable sofic shift. As corollaries of the results we obtain undecidability results for cellular automata on uncountable sofic shifts.

The `marker method' (the use of unbordered words) usually plays an important role in the construction of cellular automata, and this note makes no exception. However, of even more importance in our constructions is the concept of `conveyor belts'. This technique is implicit in many constructions found in the literature, and variants of it are found for example in \cite{BoLiRu88} and \cite{KiRo90}. We make this idea more explicit.

While our main interest is in automorphism groups, we prove results for the whole endomorphism monoid when possible, since the results about monoids are strictly stronger, and the proofs are typically the same. In some results, however, the fact we consider groups instead of monoids is important. For example, we do not know whether the set of submonoids of the endomorphism monoid of a full shift is closed under free product.

We also show some undecidability results that follow as corollaries from combining existing undecidability results about periodicity from \cite{KaOl08} with our results: in particular we obtain that given two reversible cellular automata on a full shift, it is undecidable whether they generate a free group.

\section{Definitions}

The letters $A$, $B$, $C$ and $S$ stand for finite alphabets. We write $A^*$ for the set of finite words over the alphabet $A$, and $A^+$ the set of nonempty words. The concatenation of two words $u,v$ is written as $u \cdot v$ or simply $uv$, and \emph{points} or \emph{configurations} $x \in A^\Z$ can be written as infinite concatenations
\[ x = \ldots w_{-3} \cdot w_{-2} \cdot w_{-1} \; . \; w_0 \cdot w_1 \cdot w_2 \ldots \]
where the `decimal point' in $w_{-1} \; . \; w_0$ need not be at the origin, but simply denotes some fixed position of the point.

If $u \in A^*$ is any word, write $u^R$ for the reversed word $u^R_i = u_{|u|-1-i}$.

Let $S$ be a finite alphabet. Then $S^\Z$ with the product topology is called a \emph{full shift}. We define the \emph{shift} $\sigma : S^\Z \to S^\Z$ by $\sigma(x)_i = x_{i+1}$. A \emph{subshift} is a topologically closed set $X \subset S^\Z$ satisfying $\sigma(X) = X$.

If $x \in X$ we write $w \sqsubset x$ if $\exists i: w = x_{[i,i+|w|-1]}$. We also write $w \sqsubset X$ and $u \sqsubset v$ for words $u, v$, with obvious meanings, and
\[ \B_n(X) = \{w \in S^n \;|\; w \sqsubset X\}. \]

A \emph{cellular automaton} on a subshift $X \subset S^\Z$ is a continuous function $f : X \to X$ that commutes with $\sigma$. Equivalently, it has a \emph{radius} $r \in \N$ and a \emph{local rule} $F : \B_{2r+1}(X) \to S$ such that $f(x)_i = F(x_{[i-r,i+r]})$ for all $x \in X, i \in \Z$.

For a standard reference on symbolic dynamics, see \cite{LiMa95,Ki98}.

We assume the reader is familiar with groups and monoids, but give the basic definitions to clarify our choice of boundary between properties and structure. A \emph{monoid} $M$ is a countable set together with an associative multiplication operation $(a, b) \mapsto a \cdot b = ab$ such that there is an \emph{identity element} $1_M \in M$ satisfying $a 1_M = 1_M a = a$ for all $a \in M$. A \emph{group} is a monoid $G$ where every element $g \in G$ has an inverse $g^{-1}$ satisfying $g g^{-1} = g^{-1} g = 1_G$. We think of the existence of $g^{-1}$ as simply a property of the element $g$, not as an operation. If $N$ and $M$ are monoids, a \emph{homomorphism} $\phi : N \to M$ is a map satisfying $\phi(a \cdot b) = \phi(a) \cdot \phi(b)$ and $\phi(1_N) = 1_M$

A \emph{submonoid} of a monoid $M$ is a subset of $N$ that contains the identity element $1_M$ of $M$ and is closed under multiplication, so that $N$ obtains a monoid structure from $M$ with $1_N = 1_M$, and we write $N \leq M$. More generally, we write $N \leq M$ if there is an \emph{embedding}, or an injective homomorphism $\phi : N \to M$. A bijective homomorphism is called an \emph{isomorphism}, and we write $M \cong N$ if $M$ and $N$ are isomorphic. A \emph{subgroup} is a submonoid that is a group. If $M$ is a monoid, write $\subm(M)$ for the class of isomorphism classes of its submonoids. We write $\subg(M)$ for the class of isomorphism classes of its subgroups.

The \emph{endomorphism monoid} $\End(X)$ of a subshift $X$ is the monoid of cellular automata on it under function composition, with $1_{\End(X)} = \ID_X$, the identity map on $X$. The \emph{automorphism group} $\Aut(X)$ of $X$ is the subgroup of $\End(X)$ containing the invertible elements:
\[ \Aut(X) = \{ f \in \End(X) \;|\; \exists g \in \End(X): f \circ g = g \circ f = \ID_X \}. \]
By compactness of $X$, this is precisely the set of elements of $\End(X)$ that are bijective.

Our definitions are set up so that the following holds.\footnote{For this, it is (at least a priori) important that the identity element of a subgroup of $\End(X)$ is the identity map -- thus, the identity element must be part of the structure of a monoid, rather than a property.}

\begin{lemma}
Let $G$ be a group such that $G \leq \End(X)$. Then $G \leq \Aut(X)$. In particular, if $\subm(\End(X)) = \subm(\End(Y))$ then $\subg(\Aut(X)) = \subg(\Aut(Y))$.
\end{lemma}

The following lemma is useful for defining cellular automata.

\begin{lemma}
Let $X$ be a subshift and let $Y \subset X$ be a subset such that $\sigma(Y) = Y$ and $\overline{Y} = X$. If $f : Y \to Y$ is uniformly continuous and commutes with the shift, then there is a unique continuous map $g : X \to X$ such that $g|Y = f$, and it is a cellular automaton.
\end{lemma}


Write $S(A)$ for the set of permutations of $A$, and $S(A,B)$ for the set of all bijections $c : A \to B$. The free group with $m$ generators is written $F_m$, and $F_\infty$ is the free group with a countably infinite set of generators.

The (external) \emph{direct product} of two monoids $M, N$ is the monoid $M \times N$ with operation $(a, b) \cdot (c, d) = (ac, bd)$. The \emph{free product} of monoids $M$ and $N$ is defined up to isomorphism as follows: Let $M \cong \langle 1_M, a_1, a_2, a_3, \ldots \;|\; r_1, r_2, \ldots \rangle$ be a presentation, where each $r_i$ is a relation of the form $u = v$ where $u, v \in \{a_1, a_2, \ldots\}^*$, and similarly let $N \cong \langle 1_N, b_1, b_2, \ldots \;|\; t_1, t_2 \ldots \rangle$. Then
\[ M*N \cong \langle a_1,b_1,a_2,b_2,a_3,b_3,\ldots \;|\; 1_M = 1_N, r_1, t_1, r_2, t_2, r_3, t_3, \ldots \rangle. \]

We also write $M^{k}$ for the direct product of $k$ copies of $M$. There are obvious embeddings $M^{k-1} \rightarrow M^{k}$, and their direct limit is written as $M^{\omega}$. The elements of the countable monoid $M^{\omega}$ are vectors of finite support with component values in $M$. Write $M * M$, $M^{*k}$ and $M^{*\omega}$ for the corresponding concepts for the free product.


\section{Lemmas about submonoids}

\begin{lemma}
Let $M$ and $N$ be monoids. Then $\subm(M) = \subm(N)$ if and only if $M \leq N$ and $N \leq M$.
\end{lemma}

In particular, to show that two groups have the same subgroups, we only need to show they contain each other as subgroups. We note that the set of subgroups of a group is not a complete invariant for group isomorphism: for example, free groups with different amounts of generators are non-isomorphic, but contain each other as subgroups.

In this note, we concentrate on automorphism groups of full shifts $\Aut(S^\Z)$ where $S$ is a finite alphabet of size at least $2$. It is not known when two such groups $\Aut(A^\Z)$ and $\Aut(B^\Z)$ are isomorphic -- in particular the case $|A| = 2, |B| = 3$ is open. Nevertheless, the set of subgroups is always the same:

\begin{lemma}
\label{lem:SameSubmonoids}
Let $A, B$ be alphabets of size at least $2$. Then
\[ \subm(\End(A^\Z)) = \subm(\End(B^\Z)). \]
\end{lemma}

In particular, it follows that $\subg(\End(A^\Z)) = \subg(\End(B^\Z))$ if $|A|,|B| \geq 2$. Lemma~\ref{lem:SameSubmonoids} follows directly from Lemma~\ref{lem:SoficContainsFull}, which we prove in Section~\ref{sec:SoficContainsFull}.

The interesting submonoids and groups are the infinite ones, as shown by the following result, which essentially already appears in \cite{He69}. The proof also illustrates the usefulness of Lemma~\ref{lem:SameSubmonoids}.

\begin{proposition}
\label{prop:Finites}
Let $|A| \geq 2$. Then every finite monoid embeds in $\End(A^\Z)$.
\end{proposition}

\begin{proof}
Every finite monoid embeds in the transformation monoid of a finite set $S$, and thus in $\End(S^\Z)$. The result then follows from the previous lemma. 
\end{proof}

We note that in general, if $\subm(M)$ is closed under binary direct or free products, then it is closed under the corresponding countable products.

\begin{lemma}
\label{lem:ClosureStuff}
Let $\subm(M) = \subm(N) = \subm(N')$. Then the following are equivalent:
\begin{itemize}
\item $N \times N' \in \subm(M)$,
\item $M \times M \in \subm(M)$,
\item $\subm(M)$ is closed under finite direct products.
\item $M^{\omega} \in \subm(M)$,
\item $\subm(M)$ is closed under countable direct products.
\end{itemize}
The analogous result is true for free products.
\end{lemma}

\begin{proof}
We give the proof for direct products, the case of free products being similar. The equivalence of the first two conditions and the equivalence of the last two conditions are direct, as is the fact that the last two conditions imply the first two. We show that the third and fourth condition follow from the second.


For this, suppose that $M \times M \leq M$. Let $\phi_0 : M \to M$ and $\phi_1 : M \to M$ be the embeddings giving the embedding $M \times M \leq M$, that is, $\phi_0(M) \cap \phi_1(M) = \{1_M\}$ and $\phi_0(a) \phi_1(b) = \phi_1(b) \phi_0(a)$ for all $a, b \in M$. If this holds for two maps, we say the \emph{embedding conditions} hold for them. For $w \in \{0,1\}^*$ define inductively $\phi_{0w} = \phi_0 \circ \phi_w$ and $\phi_{1w} = \phi_1 \circ \phi_w$. Then all maps $\phi_w$ are embeddings of $M$ into itself.

For $i \in \N$, define $\psi_i = \phi_{1^i0}$. It is easy to show that the embedding conditions hold for $\psi_i$ and $\psi_j$ whenever $i \neq j$. It follows that $(\psi_i)_{i \in [0,k-1]}$ gives an embedding of $M^{k}$ into $M$, and $(\psi_i)_{i \in \N}$ of $M^{\omega}$ into $M$.

The claims for free products are proved analogously, but using a different embedding conditions, namely that the images of the embeddings satisfy no nontrivial relations. 


\end{proof}

The importance of this lemma and Lemma~\ref{lem:SameSubmonoids} is that we do not need to worry about changing the alphabets of our full shifts when proving closure properties or about whether we use finite or countable products.


\section{Conveyor belts}

In our constructions, we will typically embed one automorphism group into another. In practise, this means that in the configurations of one full shift $S^\Z$, we identify subsequences that code (parts of) configurations from another full shift $A^\Z$. In these subsequences, we apply $f \in \End(A^\Z)$. To make this into a homomorphism from $\End(A^\Z)$ to $\End(S^\Z)$, it is important to have natural behavior at the boundary between an area coding (part of) a configuration in $A^\Z$, and an area containing something else. For this, we use conveyor belts.

\begin{definition}
Let $A$ be any alphabet. A \emph{conveyor belt over $A$} is a word over the alphabet $A^2$, that is, $w \in (A^2)^*$. Write $\mathrm{Conv}_{A,k}$ for the set of conveyor belts over $A$ of length $k$, that is $\mathrm{Conv}_{A,k} = (A^2)^k$. For a cellular automaton $f : A^\Z \to A^\Z$ and $k \in \N$, we define a function $f_k : \mathrm{Conv}_{A,k} \to \mathrm{Conv}_{A,k}$ as follows: if $w \in (A^2)^k$, we decompose $w$ as $w = u \times v$ for some $u, v \in A^k$, and we define
\[ f_c(w) = f((u v^R)^\Z)_{[0,k-1]} \times (f((u v^R)^\Z)_{[k,2k-1]})^R. \]
\end{definition}

Applying a CA to a conveyor belt can alternatively be described as applying its local rule on the first track, applying its local rule in reverse on the second track, and gluing the tracks at the borders in the obvious way, as if the word were laid down on a conveyor belt: it is clear that this is essentially the same as applying the CA to a periodic point of even period. The following is then clear.

\begin{lemma}
\label{lem:ConveyorBelts}
The map $f \mapsto f_k$ is a monoid homomorphism from $\End(A^\Z)$ to the monoid of functions on $\mathrm{Conv}_{A,k}$, and
\[ f \mapsto (f_k)_{k \in \N} : \End(A^\Z) \to  \prod_k S(\mathrm{Conv}_{A,k}) \]
is an embedding. Furthermore, the maps $f_k$ are uniformly continuous in the sense that there exists a \emph{radius} $r \in \N$ such that for all $k \in \N$ and $w, w' \in \mathrm{Conv}_{A,k}$ such that
\[ w_{[\max(0,i-r),\min(k-1,i+r)]} = w'_{[\max(0,i-r),\min(k-1,i+r)]}, \]
we have $f_k(w)_i = f_k(w')_i$.
\end{lemma}

The radius in the lemma can be taken to be just the usual radius of $f \in \Aut(A^\Z)$.

\section{Direct products}

\begin{theorem}
\label{thm:DirectProducts}
If $|S| \geq 2$, $\subm(\End(S^\Z))$ is closed under direct products.
\end{theorem}

\begin{proof}
By Lemma~\ref{lem:SameSubmonoids} and Lemma~\ref{lem:ClosureStuff}, it is enough to show that for two disjoint alphabets $A$ and $B$, the monoid $\End(A^\Z) \times \End(B^\Z)$ embeds into $\End(S^\Z)$ for some alphabet $S$. We choose $S = A^2 \sqcup B^2$. To prove the claim, it is enough to give embeddings $f \mapsto f_A : \End(A^\Z) \to \End(S^\Z)$ and $f \mapsto f_B : \End(B^\Z) \to \End(S^\Z)$ such that $f_A \circ g_B = g_B \circ f_A$ for all $f \in \End(A^\Z)$ and $g \in \End(B^\Z)$, and $f_A = g_B \implies f = \ID_{A^\Z} \wedge g = \ID_{B^\Z}$.

Let $f \in \End(A^\Z)$. Write $Y$ for the set of points $x \in (A^2 \sqcup B^2)$ which are not left or right asymptotic to a point over $A^2$ or $B^2$, in other words, points where all continuous runs over one of the subalphabets $A^2$ or $B^2$ are finite. If $x \in Y$, then we can write
\[ x = \ldots w_{-3} \cdot w_{-2} \cdot w_{-1} \; . \; w_0 \cdot w_1 \cdot w_2 \ldots \]
where $w_{2i} \in (A^2)^+$ and $w_{2i+1} \in (B^2)^+$ for all $i \in \Z$ (and the decimal point need not be at the origin), and we define
\[ f_A(x) = \cdots w_{-3} \cdot f_{|w_{-2}|}(w_{-2}) \cdot w_{-1} \; . \; f_{|w_0|}(w_0) \cdot w_1 \cdot f_{|w_2|}(w_2) \cdot w_3 \cdots \]
(where the decimal point is in the same position as in $x$). It is clear that $f_A$ commutes with the shift on $Y$. From Lemma~\ref{lem:ConveyorBelts}, it follows that $f_A$ is uniformly continuous on $Y$, and thus extends uniquely to a cellular automaton on $X$. This gives a function from $\End(A^\Z)$ to $\End(S^\Z)$. From Lemma~\ref{lem:ConveyorBelts}, it easily follows that this is an embedding.

We symmetrically construct a homomorphism from $\Aut(B^\Z)$ to $\Aut(S^\Z)$ by rewriting the contiguous segments over $B^2$. It is clear that these mappings commute, and since $f_A$ fixes all symbols from $B^2$ and $g_B$ all symbols from $A^2$, only the identity map is in the image of both embeddings. 
\end{proof}

\begin{corollary}
If $|S| \geq 2$, $\subg(\Aut(S^\Z))$ is closed under direct products.
\end{corollary}

\section{Free products}

In this section, we prove our main theorem:

\begin{theorem}
\label{thm:FreeProducts}
If $|S| \geq 2$, $\subg(\Aut(S^\Z))$ is closed under free products.
\end{theorem}

In the direct product case, we had two kinds of conveyor belts, and we applied the $f_A$ and $g_B$ maps completely independently on both of them. To embed a free product, the idea is to have these conveyor belts talk to each other, so that any alternating product of elements from the $(f \mapsto f_A)$-embedding and the $(g \mapsto g_B)$-embedding can transmit information arbitrarily far over an alternating sequence of $A$-conveyor belts and $B$-conveyor belts, as long as the lengths and contents of these belts are chosen suitably.

For this, we increase the size of our alphabet, and add a $C$-component that allows us to transmit any kind of modification of the left end of an $A$-conveyor belt to any kind of modification of the right end of a $B$-conveyor belt on its left, and vice versa. Let us proceed to the details.

\begin{proof}
By Lemma~\ref{lem:SameSubmonoids}, it is enough to show how to embed the free product of $\Aut(A^\Z)$ and $\Aut(B^\Z)$ into $\Aut(S^\Z)$, where $A$ and $B$ are two disjoint alphabets with the same cardinality and $S$ is an alphabet of our choosing. We choose an arbitrary abelian group structure on both $A^2$ and $B^2$. Let $C = S(A^2, B^2)$, and choose the alphabet
\[ S = C \sqcup A^2 \sqcup B^2. \]

Similarly as in the proof of Theorem~\ref{thm:DirectProducts}, for each $f \in \Aut(A^\Z)$, we will define a CA $f_A : S^\Z \to S^\Z$ so that $f \mapsto f_A$ embeds $\Aut(A^\Z)$ into $\Aut(S^\Z)$ as a subgroup $G_A \leq \Aut(S^\Z)$. An embedding $f \mapsto f_B$ of $B^\Z$ into $G_B \leq S^\Z$ is defined symmetrically by swapping the roles of $A$ and $B$, and $G_A \cap G_B = \{\ID\}$ will be clear from the construction. To show they generate the free product of $\Aut(A^\Z)$ and $\Aut(B^\Z)$ in $\Aut(S^\Z)$, we must show there are no non-trivial relations between elements of $G_A$ and $G_B$. Because the roles of $A$ and $B$ are symmetric, this will follow from showing that for any
\[ f = f_{k-1} \circ f_{k-2} \circ \cdots \circ f_1 \circ f_0 \in \Aut(S^\Z) \]
where $k \geq 1$, and $f_i \in G_A \setminus \{\ID\}$ for even $i$ and $f_i \in G_B \setminus \{\ID\}$ for odd $i$, we have $f \neq \ID$.

Let $f \in \Aut(A^\Z)$. Let $Y$ be the set of points where contiguous runs over each of the subalphabets $A^2$, $B^2$ and $C$ are finite. The set $Y$ is shift-invariant, so to define cellular automata on $S^\Z$, it is enough to define shift-invariant uniformly continuous maps on $Y$. Let thus $x \in Y$, so that
\[ x = \ldots w_{-2} \cdot w_{-1} \; . \; w_0 \cdot w_1 \cdot w_2 \ldots \]
where for all $i \in \Z$, $w_i \in (A^2)^+$, $w_i \in (B^2)^+$ or $w_i \in C^+$, and $w_i$ and $w_{i+1}$ are over different alphabets. We write
\[ f(x) = \ldots u_{-2} \cdot u_{-1} \; . \; u_0 \cdot u_1 \cdot u_2 \ldots, \]
where $|u_i| = |w_i|$ for all $i$. If $w_i \in C^+$, we let $u_i = w_i$. If $w_i \in (A^2)^+$, we apply $f$ to the conveyor belt $w_i$ as in the proof of Theorem~\ref{thm:DirectProducts}, and let $u_i = f_{|w_i|}(w_i)$.

Finally, let us define $u_i$ for $w_i \in (B^2)^+$. Suppose $|u_j| = \ell$ First, $(u_i)_j = (w_i)_j$ for $j \in [0,\ell-2]$. If $w_{i+1} \notin C$ (that is, either $w_{i+1} \notin C^+$ or $w_{i+1} \in C^+$ but $|w_{i+1}| > 1$) or $w_{i+2} \notin (A^2)^+$, we also let $(u_i)_{\ell-1} = (w_i)_{\ell-1}$. If $w_{i+2} \in (A^2)^+$ and $w_{i+1} = c \in C$, we let
\[ (u_i)_{\ell - 1} = (w_i)_{\ell - 1} - c((w_{i+2})_0) + c((u_{i+2})_0),\]
where addition is performed with respect to the abelian group structure of $B^2$. It is easy to see that this defines $f_A$ on $Y$, and since the map defined is uniformly continuous and shift-commuting, we can extend it in a unique way to a CA $f_A : S^\Z \to S^\Z$. If $f \in \Aut(A^\Z)$, it has an inverse $f^{-1} \in \Aut(A^\Z)$, and it is easy to see that $f_A|_Y \circ f_A^{-1}|_Y = f_A^{-1}|_Y \circ f_A|_Y = \ID_Y$. It follows that $f_A \in \Aut(A^\Z)$, because the only extension of the identity map on $Y$ to a CA on $S^\Z$ is the identity map.

For $g \in \Aut(B^\Z)$, the map $g_B : S^\Z \to S^\Z$ is defined symmetrically, modifying the rightmost symbol of a word over $A^2$ as a function of the leftmost symbol of a word over $B^2$ to the right of it when separated by the single symbol $c \in C$, by the same formula, but using the abelian structure of $A^2$ instead of that of $B^2$, and the function $c^{-1} : B^2 \to A^2$ instead of $c$.

We now show that $G_A$ and $G_B$ indeed give the free product of $\Aut(A^\Z)$ and $\Aut(B^\Z)$. Suppose thus that
\[ f = f_{k-1} \circ f_{k-2} \circ \cdots \circ f_1 \circ f_0 \in \Aut(S^\Z) \]
where $k \geq 1$, and $f_i \in G_A \setminus \{\ID\}$ for even $i$ and $f_i \in G_B \setminus \{\ID\}$ for odd $i$. We need to show $f \neq \ID$. Suppose the minimal radius of $f_i$ is $r_i \geq 0$. We show that $f(x)_0$ depends on at least the cell $x_h$ where $h = \sum_i r_i + 2k$. Clearly this will imply that $f \neq \ID$.

For this, we will define two points
\[ x = y \; . \; u_{k-1} \cdot c_{k-2} u_{k-2} \cdots c_1 u_1 \cdot c_0 u_0 \cdot z \]
\[ x' = y \; . \; u_{k-1} \cdot c_{k-2} u_{k-2} \cdots c_1 u_1 \cdot c_0 u_0' \cdot z \]
where $y \in C^{-\N}$, $z \in C^\N$, $u_i \in (A^2)^+$ if $i$ is even, $u_i \in (B^2)^+$ if $i$ is odd, $|u_0'| = |u_0|$, $(u_0)_i = (u_0')_i$ for $i \neq |u_0|-1$, and $c_i \in C$ for all $i$. The choices of $y$ and $z$ are arbitrary, and $|u_i| = r_i$ for all $i$. We will choose the words $u_i$ and symbols $c_i$ carefully so that $f(x)_0 \neq f(x')_0$, while $x_i = x_i'$ for all $i \neq h$. Write $I_i = [\alpha_i, \beta_i]$ for the interval where $u_i$ occurs in $x$.

Let $i \in [0,k-1]$ be even, and let $v_i$ and $v_i'$ be words of length $2r_i + 1$ with $(v_i)_j = (v_i')_j$ for $j \in [1,2r_i-1]$ such that $f_i(v_i) \neq f_i(v_i')$. We then have $(v_i)_0 \neq (v_i')_0$ or $(v_i)_{2r_i} \neq (v_i')_{2r_i}$, and if $r_i > 0$ we also suppose $(v_i)_0 = (v_i')_0$ or $(v_i)_{2r_i} = (v_i')_{2r_i}$. If $(v_i)_0 \neq (v_i')_0$, $i$ is called a \emph{left dependence} and otherwise a \emph{right dependence}.

If $i$ is a left dependence, define the words $t_i, t_i' \in (A^2)^{r_i+1}$ by
\[ (t_i)_j = \left\{\begin{array}{ll}
((v_i)_{r_i + j + 1}, (v_i)_{r_i - j}) & \mbox{if } j < r_i \\
(a, (v_i)_0) & \mbox{if } j = r_i,  \\
\end{array}\right.
\]
where $a$ is arbitrary, and $t_i'$ by the same formula using the word $v_i'$. The word $t_i$ is a conveyor belt version of $v_i$ where the center of $v_i$ is at the left end of $t_i$ and the leftmost coordinate is at the right end of $t_i$.

The important property of $t_i$ and $t_i'$ is that they agree apart from their rightmost coordinates, but after applying $f_i$ to the conveyor belt, the images differ in the leftmost coordinate, that is,
\[ (t_i)_{[0,r_i-1]} = (t_i')_{[0,r_i-1]} \wedge (f_i)_{|t_i|}(t_i)_0 \neq (f_i)_{|t_i|}(t_i')_0. \]
If $i$ is a right dependence, we produce words $t_i$ and $t_i'$ with this property in a similar fashion.

Choose $c_i$ so that $c_i((f_i)_{|t_i|}(t_i)_0) = 0$ and
\[ c_i((f_i)_{|t_i|}(t_i')_0) = (t_{i+1}')_{r_{i+1}} - (t_{i+1})_{r_{i+1}}. \]

For odd coordinates $i$, we choose words $t_i, t_i' \in (B^2)^+$ and $c_i \in C$ similarly, switching the roles of $A$ and $B$.

We let $u_0 = t_0$, $u_0' = t_0'$. The words $u_i$ will be chosen inductively in such a way that
\[ f_{i-1} \circ \cdots \circ f_1 \circ f_0 (x)_{I_i} = t_i, \]
\[ f_{i-1} \circ \cdots \circ f_1 \circ f_0 (x')_{I_i} = t_i', \]
and
\[ f_{i-1} \circ \cdots \circ f_1 \circ f_0 (x)_{I_j} = f_{i-1} \circ \cdots \circ f_1 \circ f_0 (x')_{I_j} \]
for $j > i$.

To see this is possible, observe that information travels only from right to left over $C$-symbols when maps from $\langle G_A, G_B \rangle$ are applied, and a map $f \in (G_A \cup G_B)^k$ will move information over at most $k$ symbols in $C$. Thus, when applying our maps ${g_\ell} = f_{\ell-1} \circ \cdots \circ f_1 \circ f_0$ to the points $x, x'$, we will automatically have $g_{\ell}(x)_{I_j} = g_{\ell}(x')_{I_j}$ for all $j > \ell$. It follows that it is enough to show that the $u_i$ can be chosen so that $g_{\ell}(x)_{I_\ell} = t_\ell$ for all $\ell$: by the choice of the symbols $c_i$, we will then automatically have $g_{\ell}(x')_{I_\ell} = t'_\ell$.

But naturally we can choose such words $u_i$ by induction on $i$, since each of the maps $g_\ell$ is reversible and information travels only to the left over symbols in $C$. 
\end{proof}

In \cite{BoLiRu88}, it is shown that $\Z_2 * \Z_2 * \Z_2$ embeds in $\Aut(X)$ for a full shift $X$, and thus also the two-generator free group does. More generally, it is known that every free product of finitely many finite groups embeds in $\Aut(X)$ for a full shift $X$. In \cite{KiRo90}, this is attributed to R. C. Alperin. Combining the previous theorem and Proposition~\ref{prop:Finites} gives a new proof of this result.

\begin{corollary}
If $|S| \geq 2$, every free product of finite groups embeds in $\Aut(S^\Z)$.
\end{corollary}

\section{Embeddings between endomorphism groups}
\label{sec:SoficContainsFull}

In \cite{KiRo90}, it was shown that automorphism groups of full shifts embed in those of transitive SFTs. With essentially the same proof, we show that endomorphism monoids of full shifts embed in those of uncountable sofic shifts, equivalently in ones with positive entropy. The proof relies on a number of basic properties of sofic shifts, which can be found in \cite{LiMa95,Ki98}.

\begin{lemma}
\label{lem:SoficContainsFull}
Let $X$ be any uncountable sofic shift, and $A^\Z$ any full shift. Then $\End(A^\Z) \leq \End(X)$.
\end{lemma}

\begin{proof}
It is easy to see that for a sofic shift, uncountability is equivalent to having positive entropy.

The \emph{syntactic monoid} of a subshift $X$ is the monoid whose elements are equivalence classes of words in $X$ under the equivalence $w \sim w' \iff \forall u, v: uwv \sqsubset X \iff uw'v \sqsubset X$. Sofic shifts are characterized as the subshifts with a finite syntactic monoid. If $X$ has positive entropy, then its minimal SFT cover $Y$ also does. Then $Y$ has a transitive component with positive entropy, and its image in the covering map is a positive-entropy transitive sofic subshift $Z$ of $X$.

Let $k$ be such that if $u, v \sqsubset Z$ then $uav \sqsubset Z$ for some $a$ with $|a| \leq k$ -- such $k$ exists because $Z$ is transitive and because its syntactic monoid is finite. Let $m$ be such that $Z$ contains at least $n = |A^2|$ words of length $m$ representing the same element of the syntactic monoid of $X$ -- such $m$ exists because the syntactic monoid of $X$ is finite, and because $Z$ has positive entropy. Let $w'$ be a \emph{synchronizing word} in $Z$ of some length $\ell$, that is, such that
\[ u_1 w', w' u_2 \sqsubset Z \implies u_1 w' u_2 \sqsubset Z. \]
Such a word can be found in any sofic shift \cite{Ki98}.

In every aperiodic infinite word, one can find unbordered words of arbitrary length \cite{Lo02}. Take any configuration in $Z$ which is aperiodic and where $w'$ appears syndetically, to find an unbordered word $w \sqsubset Z$ of length at least $2k+m+1$ containing the word $w'$. Then $w$ is synchronizing, since it contains a synchronizing subword.

By the assumptions on $k$ and $m$, there is a set of words $wUw$ where $U \subset S^p$ for some $p \leq {2k+m}$ and $|U| = n$ such that two words in $wUw$ can only overlap nontrivially by sharing the subword $w$, and all words in $wUw$ represent the same element of the syntactic monoid of $X$. Since all words in $wUw$ occur in $Z$ and $w$ is synchronizing, the language $(wU)^*w$ is contained in the language of $Z$.

Now, let $U = \{u_1, \ldots, u_n\}$, and fix a bijection $\phi : [1, n] \to A^2$. Given a CA $f : A^\Z \to A^\Z$, the embedding is now constructed as in the previous sections: if we have a maximal finite subword of the form
\[ w u_{i_1} w u_{i_2} w u_{i_3} w \cdots w u_{i_\ell} w \]
(note that $U$-subwords of two such words cannot overlap by the assumptions), we apply $f$ to the corresponding conveyor belt and let $v = f_{\ell}(\phi(i_1) \phi(i_2) \cdots \phi(i_\ell))$. We rewrite the word $w u_{i_1} w \cdots w u_{i_\ell} w$ by
\[ w u_{\phi^{-1}(v_1)} w u_{\phi^{-1}(v_2)} w u_{\phi^{-1}(v_3)} w \cdots w u_{\phi^{-1}(v_\ell)} w. \]

As in the previous proofs, it is easy to check that this gives an embedding of $\End(A^\Z)$ into $\End(X)$.
\end{proof}

The converse $\End(X) \leq \End(A^\Z)$ is not true in general for positive-entropy sofic shifts $X$:

\begin{example}
The group $\Aut(A^\Z)$ is residually finite \cite{BoLiRu88}. We show that $\End(X)$ need not be residually finite for a sofic shift $X$. For this, let
\[ X_k = \{x \in \{0,1\}^\Z \;|\; |\{i \in \Z \;|\; x_i = 1\}| \leq k\}. \]
Then $\Aut(X_2)$ contains a copy of the group of all permutations of $\N$ with finite support, by permuting the (orbits of) isolated points. This group is not residually finite, so $X_2 \times Y$ is not residually finite for any subshift $Y$. In particular, by letting $Y$ be a positive-entropy sofic shift we obtain the result. \qee
\end{example}

Even assuming transitivity, we are not aware of a general technique of embedding $\End(X)$ into $\End(Y)$ for two sofic shifts, and this seems tricky to do even for particular examples.

\begin{question}
Let $X$ and $Y$ be two mixing SFTs. When do we have $\Aut(X) \leq \Aut(Y)$?
\end{question}

\section{Decidability}


In this section, we briefly discuss some decidability corollaries for cellular automata that follow from the constructions. We fix the local rule of a CA as its computable presentation. This allows us to ask decidability questions about cellular automata. We start with a lemma that shows that the translation between cellular automata and their local rules is completely algorithmic. We omit the standard proof.


\begin{lemma}
Let $X$ be a sofic shift. Then given a function $F : S^{2r+1} \to S$, it is decidable whether $f(x)_i = F(x_{[i-r,i+r]})$ defines a cellular automaton on $X$, and if it does, we can compute a minimal radius $r' \in \N$ and a local rule $F' : \B_{2r'+1}(X) \to S$ for $f$.
\end{lemma}

Of course, an algorithm that minimizes the local rule implies that given two local rules $F : S^{2r+1} \to S$ and $F' : S^{2r'+1} \to S$, it is decidable whether they define the same CA.

Let $G$ be a countable group, with a fixed computable presentation for the elements. The \emph{torsion problem} is the problem of, given $g \in G$, deciding whether there exists $m > 0$ such that $g^m = 1_G$.

\begin{theorem}[\cite{Ka92}]
For some $S$, the torsion problem of $\Aut(S^\Z)$ is undecidable.
\end{theorem}

\begin{theorem}[\cite{SaTo12c}]
The torsion problem of $\Aut(X)$ is decidable for every zero-entropy sofic shift $X$.
\end{theorem}

Combining the theorems with Lemma~\ref{lem:SoficContainsFull} gives the following.

\begin{theorem}
Let $X$ be a sofic shift. Then the torsion problem of $\Aut(X)$ is decidable if and only if $X$ has zero entropy.
\end{theorem}

Another definition of the torsion problem is to decide, given a CA, whether it generates a copy of $\Z$. Next, we discuss other problems of this type, omitting the easy proofs.

\begin{proposition}
Given a finite set $F \subset \Aut(S^\Z)$ and a finite group $G$, it is decidable whether $\langle F \rangle \cong G$.
\end{proposition}

Abelianness is also easy to check.

\begin{proposition}
Given a finite set $F \subset \Aut(S^\Z)$, it is decidable whether $\langle F \rangle$ is abelian.
\end{proposition}

Nevertheless, combining the result of \cite{Ka92} with Theorem~\ref{thm:DirectProducts} and the fundamental theorem of abelian groups we see that it is impossible to check \emph{which} abelian group is generated by a finite set of CA. 

\begin{proposition}
Let $G$ be any infinite finitely-generated abelian group. Then given a finite set $F \subset \Aut(S^\Z)$, it is undecidable whether $\langle F \rangle \cong G$.
\end{proposition}

Finally, combining the result of \cite{Ka92} with Theorem~\ref{thm:FreeProducts}, we obtain undecidability of free products.

\begin{theorem}
Given $f, g \in \Aut(S^\Z)$, it is undecidable whether $\langle f, g \rangle \cong F_2$.
\end{theorem}

\section{Acknowledgements}

The author was supported by FONDECYT Grant 3150552.

\bibliographystyle{plain}
\bibliography{../../bib/bib}{}

\end{document}